\documentclass[reqno,11pt]{amsart}
\usepackage[margin=1in]{geometry}
\usepackage[T1]{fontenc}
\usepackage[utf8]{inputenc}
\usepackage{setspace}

\usepackage{amsthm, amsmath, amssymb,xcolor,mathtools}

\usepackage[textsize=small,backgroundcolor=orange!20]{todonotes}

\usepackage[hidelinks]{hyperref}
\usepackage[shortlabels]{enumitem}

\usepackage[noabbrev,capitalize]{cleveref}
\crefformat{equation}{(#2#1#3)}
\crefmultiformat{equation}{(#2#1#3)}{ and~(#2#1#3)}{, (#2#1#3)}{ and~(#2#1#3)}

\usepackage[color,final]{showkeys} 

\colorlet{refkey}{orange!20}
\colorlet{labelkey}{blue!60}

\numberwithin{equation}{section}

\newtheorem{theorem}{Theorem}[section]	
\newtheorem{proposition}[theorem]{Proposition}
\newtheorem{lemma}[theorem]{Lemma}

\newtheorem{corollary}[theorem]{Corollary}

\theoremstyle{definition}
\newtheorem{definition}[theorem]{Definition}

\theoremstyle{remark}
\newtheorem{remark}[theorem]{Remark}

\newcommand{\abs}[1]{\left\lvert#1\right\rvert}

\newcommand{\norm}[1]{\left\lVert#1\right\rVert}

\newcommand{\ang}[1]{\left\langle #1 \right\rangle}
\newcommand{\angs}[1]{\langle #1 \rangle}
\newcommand{\floor}[1]{\left\lfloor #1 \right\rfloor}
\newcommand{\ceil}[1]{\left\lceil #1 \right\rceil}
\newcommand{\paren}[1]{\left( #1 \right)}
\newcommand{\sqb}[1]{\left[ #1 \right]}
\newcommand{\set}[1]{\left\{ #1 \right\}}

\newcommand{\wt}{\widetilde}

\DeclareMathOperator{\im}{im}
\DeclareMathOperator{\cl}{cl}
\DeclareMathOperator{\proj}{proj}

\newcommand{\CC}{\mathbb{C}}
\newcommand{\EE}{\mathbb{E}}

\newcommand{\RR}{\mathbb{R}}
\newcommand{\PP}{\mathbb{P}}

\newcommand{\ZZ}{\mathbb{Z}}

\newcommand{\net}{\mathrm{net}}

\title[Spectral non-concentration near the top]{Spectral non-concentration near the top \\ for unimodular random graphs}

\author[Fraczyk]{Mikolaj Fraczyk}
\author[Hayes]{Ben Hayes}
\author[Sudan]{Madhu Sudan}
\author[Zhao]{Yufei Zhao}

\thanks{Hayes was supported in part by NSF grant DMS-2000105. 
Fraczyk was supported in part by the Dioscuri programme initiated
by the Max Planck Society, jointly managed with the National Science Centre in Poland, and mutually funded by
Polish the Ministry of Education and Science and the German Federal Ministry of Education and Research.
Sudan was supported in part by a Simons Investigator Award and CCF 2152413.
Zhao was supported in part by NSF CAREER award DMS-2044606.}

\address{Fraczyk: Faculty of Mathematics and Computer Science, Jagiellonian University, ul. Łojasiewicza 6, 30-348 Krak{\'o}w, Poland}
\email{mikolaj.fraczyk@uj.edu.pl}

\address{Hayes: Department of Mathematics,
University of Virginia, 141 Cabell Drive, Kerchof Hall, Charlottesville, VA, 22904}
\email{brh5c@virginia.edu}

\address{Sudan: School of Engineering and Applied Sciences, Harvard University, Cambridge, Massachusetts, USA}
\email{madhu@cs.harvard.edu}

\address{Zhao: Massachusetts Institute of Technology, Cambridge, MA, USA}
\email{yufeiz@mit.edu}

\begin{document}

\begin{abstract}
In recent work on equiangular lines, Jiang, Tidor, Yuan, Zhang, and Zhao showed that a connected bounded degree graph has sublinear second eigenvalue multiplicity. 
More generally they show that there cannot be too many eigenvalues near the top of the spectrum.
We extend this result to infinite unimodular random graphs.
As a corollary, the spectral distribution of the adjacency operator cannot have an atom at the top. 
For an infinite regular expander, we deduce that the singularity of the spectral measure at the top satisfies $\mu_G[(1-\theta)\rho,\rho] \lesssim \theta^c$ for some constant $c>0$, where $\rho$ is the spectral radius of the adjacency operator of the graph. This  implies new general estimates on the return probabilities of random walks.
\end{abstract}

\maketitle

\section{Introduction}

The shapes of spectral distributions of graphs is an important topic.
Classic results include the Kesten--McKay law~\cite{Kes59,McK81} giving the spectral distribution of a random $d$-regular graph, 
the Alon--Boppana bound~\cite{Alo86,Nil91} on the minimum possible second eigenvalue of a $d$-regular graph,
and results showing the optimality of the Alon--Boppana bound including Friedman's theorem~\cite{Fri08} and Ramanujan graphs~\cite{LPS88,Mar88,MSS15}. 
See the survey \cite{HLW06} on expander graphs.

Our goal is to give some interesting necessary conditions on the spectral distribution of bounded degree graphs, both for finite and infinite graphs. 
Informally speaking, we show that it is impossible for significant mass of the spectral distribution to be located near the top of the spectrum of the adjacency operator.
In this article, we are always referring to the spectrum of the adjacency operator, as opposed to the Laplacian. For regular graphs, the spectrum of the Laplacian and the adjacency operators coincide up to an easy transformation.

Jiang, Tidor, Yao, Zhang, and Zhao~\cite{JTYZZ21} showed that a connected bounded degree graph has sublinear second eigenvalue multiplicity.
This spectral result was a key step in solving the longstanding problem of determining the maximum number of equiangular lines with a fixed angle in high dimensions.
The dependence on $\Delta$ below was not made explicit in \cite{JTYZZ21} but can be read from the proof.

\begin{theorem}[{\cite{JTYZZ21}}] \label{thm:2nd-eig}
For a connected $n$-vertex graph with maximum degree at most $\Delta$,
the second largest eigenvalue of its adjacency matrix has multiplicity $O((\log \Delta) n/\log\log n)$.	
\end{theorem}

In this paper, we first explain how to modify the proof of \cref{thm:2nd-eig} from \cite{JTYZZ21} to obtain a statement about the non-concentration of the spectral distribution near the top in finite connected bounded-degree graphs.
Then, we extend the result from finite graphs to a class of infinite graphs known as unimodular random graphs.
The proof for unimodular random graphs requires several new ingredients, including an extension of the Cauchy eigenvalue interlacing theorem and a net selection lemma inspired by ideas from distributed local graph algorithms (e.g., \cite{NO08}).
We leave it as an open problem to extend the result to additional settings beyond unimodularity, such as stationary random graphs~\cite{BC12}.

Let us comment on the various hypotheses in \cref{thm:2nd-eig}.
We cannot drop the connectivity hypothesis, since otherwise we can take a disjoint union of $m$ copies of some fixed connected graph, which has the top eigenvalue repeated $m$ times.
The maximum degree hypothesis also cannot be dropped, since otherwise we can take the complete graph, or some strongly regular graph, which has linear second eigenvalue multiplicity. Though, it is unclear how sharp is the $\log \Delta$ factor in the theorem.
Finally, it is important that we are looking at eigenvalues near the top, as there exist connected bounded degree graphs whose most negative eigenvalue has linear multiplicity \cite[Example 8.3]{JTYZZ2}. 

There are several ways to generalize \cref{thm:2nd-eig}.
For example, \cite{JTYZZ21} also showed that instead of the second largest eigenvalue, the $j$-th largest eigenvalue has multiplicity $O_{\Delta, j}(n/\log\log n)$. 
Furthermore, a similar bound also applies to the number of eigenvalues (counting multiplicity) within a small window of the $j$-th largest eigenvalue.
See \cref{thm:finite-main} below for precise statements.

The proof of \cref{thm:2nd-eig} also allows positive edge weights in some fixed interval $[w_{\min}, w_{\max}]$ with $0 < w_{\min} \le w_{\max}$. 
Given an edge-weighted graph $G$, we define $A_G$ to be the edge-weighted adjacency matrix of $G$, whose $(u,v)$-entry is the edge-weight of $G$ if $uv \in E(G)$ and zero otherwise.
The reader is welcomed to ignore edge-weights and focus on the setting of unit edge weights for simplicity. 
We include this generalization to edge weights since it captures the normalized adjacency matrix $D_G^{-1/2} A_G D_G^{-1/2}$ associated to the simple random walk on an unweighted $G$, where $D_G$ is the diagonal matrix of vertex degrees.
This is equivalent to weighing every edge $uv$ of the unweighted $G$ by $(\deg u)^{-1/2}(\deg v)^{-1/2}$ to obtain an edge weighted $\wt G$, so that $A_{\wt G} = D_G^{-1/2} A_G D_G^{-1/2}$.
McKenzie, Rasmussen, and Srivastava~\cite{MRS21} showed that the second eigenvalue multiplicity of $A_{\wt G}$ is $O_\Delta(n/(\log n)^c)$ for a connected $n$-vertex graph $G$ of maximum degree $\Delta$, which improves over the $O_\Delta(n/\log\log n)$ bound in \cite{JTYZZ2} for reguar graphs.

Note that when we talk about the maximum degree $\Delta$ or distance in $G$, 
we are always referring to that of the underlying unweighted graph (i.e., ignoring the edge-weights).

We write $\abs{G}$ for the number of vertices of $G$.
For any $E \subseteq \RR$, write $m_G(E)$ for the number of eigenvalues of the weighted adjacency matrix of $G$ lying in $E$.
Let us write $\mu_G$ to be the spectral distribution of $G$ defined by $\mu_G(E) := m_G(E) / \abs{G}$.
To simplify notation, we write $\mu_G(a,b]$ instead of $\mu_G((a,b])$ and likewise with other types of intervals.
We write $x \lesssim_A y$ to mean $x \le C_A y$ for some constant $C_A$ that depends only on $A$.

In \cref{sec:infinite} we recall the notion of a unimodular random graph, which is a useful notion of infinite graphs introduced by Benjamini and Schramm~\cite{BS01} and further developed by Aldous and Lyons~\cite{AL07}.
The spectral distribution $\mu_G$ will also be defined in \cref{sec:infinite}.

Let us first state our main results in a qualitative form. 

\begin{theorem} \label{thm:intro}
	For every $\Delta, \epsilon > 0$ and $0 < w_{\min} \le w_{\max}$,
	there exist $\theta, \delta > 0$ such that the following holds.
	Let $G$ be an edge-weighted connected graph with at least $1/\delta$ vertices or an infinite connected unimodular random edge-weighted graph.
	Assume that $G$ has maximum degree at most $\Delta$, and all edge-weights lie in $[w_{\min}, w_{\max}]$.
	If $x > 0$ satisfies $\mu_G(x,\infty) \le \delta$, then $\mu_G[(1-\theta)x, x] \le \epsilon$.
\end{theorem}

We make precise the quantitative dependencies in \cref{sec:bounds}. Here is a preview.
Given $\Delta, w_{\max}, w_{\min}$ as in \cref{thm:intro}, there exists a constant $C > 0$ such that
\[
\mu_G[(1-\theta)x,x] \lesssim 1/\log(1/\theta)
\quad \text{whenever }\mu_G(x,\infty) \le  e^{- C/\theta} \text{ and } n > C \log (1/\theta).
\]
Furthermore, if the graph is an expander, then the bound $\mu_G[(1-\theta)x,x] \lesssim 1/\log(1/\theta)$ can be improved to 
\[
\mu_G[(1-\theta)x,x] \lesssim \theta^c.
\]
for some constant $c>0$ that depends only on $\Delta, w_{\max}, w_{\min}$ and the expansion ratio.

In particular, in the infinite setting, this rules out the possibility of having an atom at the top.

\begin{corollary}[No atom at the top] \label{cor:atom-top}
	Let $G$ be an infinite connected edge-weighted unimodular random graph with bounded degrees and edge weights lying in $[w_{\min}, w_{\max}]$ for some $0 < w_{\min} \le w_{\max}$.
	If $x > 0$ satisfies $\mu_G(x,\infty) = 0$, then $\mu_G(\set{x}) = 0$.
\end{corollary}

Quantitatively, in the setting of \cref{cor:atom-top},
\[
\mu_G[(1-\theta)x,x] \lesssim 1/\log(1/\theta).
\]
For expanders, 
\[
\mu_G[(1-\theta)x,x] \lesssim \theta^c,
\]
where $c > 0$ depend on the degree bound, edge weight bounds, and expansion ratio. See \cref{sec:bounds} for precise statements.

\begin{remark}
It is possible to have an atom at the bottom of the spectrum \cite[Example 8.3]{JTYZZ2}.    
\end{remark}

If a unimodular random graph $G$ is the local limit of finite graphs $G_n$, then the spectral distribution converges $\mu_{G_n}(I) \to \mu_{G}(I)$ on every interval $I$, as shown by Ab\'ert, Thom, and Vir\'ag~\cite{ATV}.
However, it is a major open problem whether every unimodular random graph can be obtained as a local limit of finite graphs. Specializing to Cayley groups, a positive solution to this open problem would prove that all groups are \emph{sofic} (itself a major open problem). A minor weakening of the soficity problem for groups is the problem of whether every group is \emph{hyperlinear} (meaning that it has an asymptotically injective sequence of approximate homomorphisms with values in the unitary group). This problem on hyperlinearity of groups is a special case of the Connes' embedding problem (first posed in \cite{Connes}) by \cite{Radul}. A preprint claiming a negative solution to the Connes embedding problem was recently announced in \cite{MIP*=RE}. It is unclear whether the methods of \cite{MIP*=RE} can be used to produce non-hyperlinear or non-sofic groups.

\begin{remark}[Quantitative bounds] \label{rem:bounds}
There are many open questions about quantitative bounds.
Here we consider the setting of fixed maximum degree $\Delta$.
We do not know if the bound $O_\Delta(n/\log\log n)$ in \cref{thm:2nd-eig} can be improved.
Haiman, Schildkraut, Zhang, and Zhao \cite{HSZZ22} gave a lower bound construction with second eigenvalue multiplicity $\Omega(\sqrt{n/\log n})$.
Also, they constructed connected bounded degree graphs with $O(n/\log\log n)$ eigenvalues within $O(1/\log n)$ of the second largest eigenvalue.

For the normalized adjacency matrix $D_G^{-1/2} A_G D_G^{-1/2}$ of a graph $G$, 
 McKenzie, Rasmussen, and Srivastava~\cite{MRS21} proved that the second largest eigenvalue multiplicity is at most $\Delta^{7/5} n (\log\log n)^{O(1)} / (\log n)^{1/5}$, which, for fixed $\Delta$, improves the dependence on $n$ from $n/\log\log n$ to $n /(\log n)^c$.
It is unclear whether ideas from \cite{MRS21} can be adapted to the adjacency matrix.

A related problem is to improve the quantitative dependencies among the parameters in \cref{thm:intro}, which are made precise in \cref{sec:bounds}.
\end{remark}

Our work has also interesting applications to the spectrum of Cayley graphs as well as the decay of the return probabilities in non-amenable groups. 
After restricting our attention to non-amenable regular unimodular random graphs we prove a stronger variant of Theorem \ref{thm:unimod-main}. 
As we discuss in Section \ref{sec:non amenable}, our results are optimal in that one can produce examples of unimodular random graphs with a polynomial lower bound on $\mu_G[\rho(1-\theta),\rho]$ where $\rho$ is the spectral radius. In fact, Cayley graphs of free groups or hyperbolic groups give examples. See Section \ref{sec:non amenable} for how this is related to the \emph{rapid decay property} for groups. 
This result yields an interesting estimate on the return probabilities of a random walk on the Cayley graphs of non-amenable groups. 

\begin{theorem}\label{thm:decay on return probs intro}
     Let $\Gamma$ be a non-amenable group generated by a finite symmetric set $S$. Let $p_{2n}$ be the probability the the standard random walk on the Cayley graph ${\rm Cay}(\Gamma,S)$ returns to the origin at time $2n$. Let $\rho:=\lim_{n\to\infty} dp_{2n}^{1/2n}$. There exists a constant $C = C_{\Gamma, S} > 0$ such that \[p_{2n}\lesssim_{\Gamma, S} (\rho/d)^{2n} n^{-C}.\]   
\end{theorem}
In \cite[Corollary 17.8]{WolfgangRW}, similar results to the above Theorem are given for the special case of free products. Namely, suppose that $(\Gamma_{i})_{i\in I}$ are finitely generated groups and that $I$ is a finite set. Let $\Gamma=\ast_{i\in I}\Gamma_{i}$ be a free product with each $\Gamma_{i}$ nontrivial and that either $|I|\geq 3$ or that $|I|=2$ and $\prod_{i\in I}(|\Gamma_{i}|-2)\geq 1$ (this is equivalent to assuming that $\Gamma$ is nonamenable). Let $S_{i}$ be finite, symmetric generating sets and let $S=\bigcup_{i}S_{i}$. Then for the standard random walk on ${\rm Cay}(\Gamma,S)$ we have that 
\[p_{2n}\lesssim \rho^{2n}n^{-3/2}\]
if each $\Gamma_{i}$ has polynomial growth of degree at most $4$. If $|I|=2$ and $\Gamma_{i}=\ZZ^{d}$ with $d\geq 5$, then  with $S_{i}$ the standard generating set of $\ZZ^{d}$ we have  by \cite[Theorem 17.13]{WolfgangRW} that
\[p_{2n}\thicksim \rho^{2n}n^{-d/2}.\]

\emph{Notation and definitions.}
We write $G[U]$ for the subgraph of $G$ induced by the vertex subset $U$. 
We write $B_G(v,r)$ for the subgraph induced by all vertices at distance at most $r$ from $v$ (we ignore edge-weights when computing distance in $G$).
An \emph{$r$-net} of $G$ is a subset $W$ of vertices such that every vertex of $G$ is within distance $r$ of $W$.
A set of vertices is \emph{$s$-separated} if every pair of vertices is separated by at least $s$ edges.
For a finite graph $G$, we write $\lambda_1(G)$ for the largest eigenvalue of $A_G$.

\subsection*{Organization}
In \cref{sec:finite} we prove the result for finite graphs, following \cite{JTYZZ21}. 
In \cref{sec:infinite} we recall the notion of a unimodular random graph and state the results for the infinite setting.
In \cref{sec:interlacing} we prove a Cauchy eigenvalue interlacing theorem for unimodular random graphs.
In \cref{sec:local} we prove some lemmas about locally selecting a random vertex subset.
In \cref{sec:unimod-pf} we prove the main results for unimodular random graphs.

\section{Quantitative bounds on spectral measure near the top} \label{sec:bounds}

Here we state the quantitative bounds for eigenvalue multiplicity in \cref{thm:intro}.
In the theorem statements, $\lesssim$ hides an absolute constant factor.

\begin{theorem}[Finite graphs] \label{thm:finite-main}
	Let $G$ be a connected $n$-vertex edge-weighted graph with maximum degree at most $\Delta \ge 2$
	 and all edge-weights in the interval $[w_{\min}, w_{\max}]$ with $0 < w_{\min} \le  w_{\max}$.
	Define $\wt \Delta = \Delta w_{\max}/w_{\min}$.
    Let $x, \theta > 0$.
	Then,
	\[
	\mu_G[(1-\theta) x, x] \lesssim \frac{1}{\log_{\wt\Delta} (1/\theta)}
	\quad\text{whenever }
	\mu_G(x, \infty) \le \wt\Delta ^{-10/\theta}
         \text{ and } 
         n \ge \log_{\wt\Delta} (1/\theta).
	\]
\end{theorem}

To deduce the $O(n\log \Delta/\log\log n)$ upper bound on second eigenvalue multiplicity (\cref{thm:2nd-eig}), apply \cref{thm:finite-main} with $x = \lambda_2(G)$ and $\theta = 10/\log_{\wt\Delta} n$ to get 
$\mu_G[(1-\theta)x, x] \lesssim n/\log_{\wt\Delta}\log_{\wt\Delta}$.

We have the following improved eigenvalue multiplicity bound for expander graphs.
We say that an $n$-vertex graph $G$ is a \emph{$c$-expander} if for every $S \subseteq V(G)$ with $\abs{G} \le n/2$, there are at least $c \abs{S}$ vertices not in $S$ but adjacent to some vertex in $S$.

\begin{theorem}[Finite expander graphs] \label{thm:finite-expander}
	Let $G$ be a connected $n$-vertex edge-weighted graph with maximum degree at most $\Delta \ge 2$
	 and all edge-weights in the interval $[w_{\min}, w_{\max}]$ with $0 < w_{\min} \le  w_{\max}$.
  Suppose $G$ is a $c$-expander for some $0 < c \le 1$.
	Define $\wt \Delta = \Delta w_{\max}/w_{\min}$.
	Let $x, \theta > 0$.
	Then
	\[
	\mu_G[(1-\theta) x, x] \lesssim \theta^{c/(40\log \wt\Delta)}
	\quad\text{whenever }
        \mu_G(x, \infty) \le \wt\Delta^{-10/\theta}
         \text{ and } 
         n \ge 2\theta^{-c/(10\log \wt\Delta)}.
	\]
\end{theorem}

In particular, setting $x = \lambda_2(G)$ and $\theta = 10/\log_{\wt\Delta} n$ as earlier gives us 
$\mu_G[(1-\theta) x, x] \lesssim 1 / (\log_{\wt\Delta} n)^{c/(40\log \wt\Delta)}$.

Next we state our results for unimodular random graphs. See \cref{sec:infinite} for definitions.

\begin{theorem}[Unimodular random graphs] \label{thm:unimod-main}
	Let $(G,o)$ be an infinite connected edge-weighted unimodular random graph with maximum degree at most $\Delta$
	 and all edge-weights in the interval $[w_{\min}, w_{\max}]$ with $0 < w_{\min} \le  w_{\max}$.
	Define $\wt \Delta = \Delta w_{\max}/w_{\min}$.
	Let $x, \theta> 0$. Then
	\[
	\mu_G[(1-\theta) x, x] \lesssim \frac{1}{\log_{\wt\Delta} (1/\theta)}
	\quad\text{whenever }
	\mu_G(x, \infty) \le \wt \Delta^{-10/\theta}.
	\]
\end{theorem}

We also have a stronger bound for expander graphs. 
We say that unimodular random graph $(G,o)$ is a \emph{$c$-expander} if almost surely for every finite $S \subseteq V(G)$, there are at least $c \abs{S}$ vertices not in $S$ but adjacent to some vertex in $S$.

\begin{theorem}[Expander unimodular random graphs] \label{thm:unimod-expander}
    Let $(G,o)$ be an infinite connected edge-weighted unimodular random graph with maximum degree at most $\Delta$
	 and all edge-weights in the interval $[w_{\min}, w_{\max}]$ with $0 < w_{\min} \le  w_{\max}$.
    Suppose $(G,o)$ is a $c$-expander for some $0 < c \le 1$.
	Define $\wt \Delta = \Delta w_{\max}/w_{\min}$.
	Let $x, \theta > 0$.
     Then
	\[
	\mu_G[(1-\theta) x, x] \lesssim \theta^{c/(40\log \wt\Delta)}
	\quad\text{whenever }
	\mu_G(x, \infty) \le \wt\Delta^{-10/\theta}.
    \]
\end{theorem}

Next, we give an improved bound when $G$ is an infinite regular expander and $x$ above is the spectral radius $\rho \coloneqq \norm{A_G}$ of the adjacency operator of $G$.
The \emph{non-amenability} condition in the next theorem is equivalent to $\rho<d$ (see \cite[Theorem 51]{CSGDHAmena}).

\begin{theorem}[Top of the spectrum for infinite regular expanders] \label{thm:RegularExp}
    Let $(G,o)$ be a non-amenable infinite connected $d$-regular unimodular random graph. Let $\rho \coloneqq \norm{A_G}$ be the spectral radius of the adjacency operator of $G$. Then for every $\theta>0$ and every $\varepsilon>0$ we have
    \[ \mu_G[\rho(1-\theta),\rho]\lesssim_{d, \rho, \varepsilon} \theta^{\frac{\log d}{\log \rho}-1-\varepsilon}.\]
\end{theorem}

\section{Finite graphs} \label{sec:finite}

We prove \cref{thm:finite-main,thm:finite-expander} by adapting \cite{JTYZZ21}.
These theorems follow from the next non-asymptotic result applied with suitable parameters.

\begin{proposition}
	 \label{prop:finite-param}
	Let $\Delta \ge 2$.
	Let $G$ be a connected $n$-vertex edge-weighted graph with maximum degree at most $\Delta$ and all edge weights in the interval $[w_{\min}, w_{\max}]$ with $0 < w_{\min} \le  w_{\max}$.
	Let $r$ and $s$ be positive integers.
	Suppose $G$ has an $r$-net with $\epsilon_\net n$ vertices.
	Let $x \ge w_{\min}$ and $\delta = \mu_G(x,\infty)$.
	Let $\theta \in [0,1)$.
	Then
	\[
	\mu_G[(1 - \theta)x, x]
	\le
	(1-\theta)^{-2s} \paren{1 - \paren{\frac{w_{\min}}{x}}^{2r}}^{s/r}
	 + 2\delta \Delta^{2(s+2)}
	 + 2 \epsilon_\net .
	\]
\end{proposition}

A key idea from \cite{JTYZZ21} is to remove a net from the graph and observe that net removal significantly reduces spectral radius
We then apply the moment method to bound eigenvalues multiplicity.

\begin{lemma}[Net removal lowers spectral radius {\cite[Lemma 4.3]{JTYZZ21}}] 
	\label{lem:rad-drop}
	Let $G$ be a non-empty edge-weighted graph with edge weights in $[w_{\min}, w_{\max}]$  with $0 < w_{\min} \le  w_{\max}$. 	Let $S \subseteq V(G)$ be an $r$-net of $G$.
	Let $H = G - S$.
	Then
	\[
	\lambda_1(H)^{2r} \le \lambda_1(G)^{2r} - w_{\min}^{2r}.
	\]
\end{lemma}

\begin{proof}
	We can assume that $G$ has no isolated vertices (removing them does not affect the claim). 
	We can view $A_H$ as obtained from $A_G$ by zeroing out the rows and columns corresponding to $S$ so that they have the same dimensions. We claim that
	\[
	A_H^{2r} \le A_G^{2r} - w_{\min}^{2r} I \qquad \text{entrywise},
	\]
	where $I$ is the identity matrix.
	To prove the entrywise matrix inequality, note that since $A_H \le A_G$ entrywise with all nonnegative entries, one has $A_H^{2r} \le A_G^{2r}$ entrywise. So it suffices to check $A_H^{2r} \le A_G^{2r} - w_{\min}^{2r} I$ on the diagonal entries. 
	For each $v \in V(G)$, the entry $(v,v)$ of $A_G^{2r}$ corresponds to summing weighted length $2r$ closed walks in $G$ starting and ending at $v$ (here the weight of a walk is defined as the product of its edge-weights), and likewise with $A_H^{2r}$ and $H$. 
	For each $v \in V(G)$, there is a length $2r$ closed walk starting at $v$ that is present in $G$ but not available in $H$, namely the walk from $v$ to its closest vertex in $S$ (which must be at distance $\le r$ due to $S$ being an $r$-net) and then back to $v$, and then further walking back and forth along some edge of $G$ to complete a length $2r$ closed walk (here we are using that $G$ has no isolated vertices).
	It follows that for each $v \in V(G)$, the $(v,v)$-entry of $A_G^{2r} - A_H^{2r}$ is at least $w_{\min}^{2r}$.
\end{proof}

\begin{lemma}[Local-global spectral comparison {\cite[Lemma 4.4]{JTYZZ21}}] 
	\label{lem:local-global}
	Let $G$ be an edge-weighted graph with nonnegative edge-weights. 
	Let $r$ be a positive integer. Then
	\[
	\sum_{i = 1}^{\abs{G}} \lambda_i(G)^{2r} \le \sum_{v \in V(G)} \lambda_1(B_G(v, r))^{2r}.
	\]
\end{lemma}

\begin{proof}
	The left-hand side above is the trace of $A_G^{2r}$, which comes from summing the weights of closed length $2r$ walks in $G$. Every length $2r$ walk starting and ending at $v$ must stay within distance $r$ of $v$, and so their contributions sum to $\angs{1_v, A_G^{2r} 1_v} = \angs{1_v, A_{B_G(v, r)}^{2r} 1_v} \le \lambda_1(B_G(v, r))^{2r}$, with the final step using the Perron--Frobenius theorem. Summing over $v$ yields the result.
\end{proof}

\begin{proof}[Proof of \cref{prop:finite-param}]
Let
\[
U = \set{ v \in V(G) : \lambda_1 (B_G(v, s+1)) > x}.
\]
Let $U_0$ be a maximal $2(s+2)$-separated subset of $U$.
Then every vertex of $U$ lies within distance less than $2(s+2)$ of $U_0$ (else we can add a new vertex to $U_0$).
Since $\abs{B_G(v, R-1)} \le \Delta^R$ for every $R$,
$\abs{U} \le \abs{U_0} \Delta^{2(s+2)}$.

Let $U_1$ be the set of all vertices in $G$ with distance at most $s+1$ to $U_0$.
Then $G[U_1]$ consists of $\abs{U_0}$ connected components each with spectral radius greater than $x$, and hence $G$ has at least $\abs{U_0}$ eigenvalues larger than $x$.
Thus $\abs{U_0} \le m_G(x,\infty) = \delta n$, and so
\[
\abs{U} \le \abs{U_0} \Delta^{2(s+2)} \le \delta \Delta^{2(s+2)}n.
\] 

Let $W$ be an $\epsilon_\net n$-vertex $r$-net of $G$.
Let 
\[
H = G - U - W
\]
 (this means removing from $G$ the vertices in $U \cup W$ and their incident edges).
 
For all $v \in V(H)$ and vertex $u$ in $B_H(v,s)$, there is some $w \in W$ with $d_G(u,w) \le r$ due to $W$ being an $r$-net in $G$. If $d_G(w,v) > s$, then the path from $u$ to $w$ in $G$ contains some $u'$ with $d_G(v,u') = s+1$.
So the vertex complement of $B_H(v,s)$ in $B_G(v,s+1)$ is an $r$-net of $B_G(v,s+1)$.
Thus by \cref{lem:rad-drop}, for all $v \in H$,
\[
\lambda_1(B_H(v, s))^{2r}
\le
\lambda_1(B_G(v, s+1))^{2r} - w_{\min}^{2r}
\le x^{2r} - w_{\min}^{2r}.
\]
Applying \cref{lem:local-global},
\[
m_H[(1 - \theta)x, x] (1-\theta)^{2s} x^{2s} 
\le
\sum_{i=1}^{\abs{H}} \lambda_i(H)^{2s}
\le
\sum_{v \in V(H)} \lambda_1( B_H(v, s))^{2s}
\le 
(x^{2r} - w_{\min}^{2r})^{s/r} n.
\]
Thus
\[
m_H[(1 - \theta)x, x] 
\le 
(1-\theta)^{-2s} \paren{1 - \paren{\frac{w_{\min}}{x}}^{2r}}^{s/r}n.
\]
By the Cauchy eigenvalue interlacing theorem,
\begin{align*}
\frac{m_G[(1 - \theta)x, x]}{n}
&\le 
\frac{m_H[(1 - \theta)x, x]}{n} + \frac{2(\abs{U} + \abs{W})}{n}
\\
&
\le
(1-\theta)^{-2s} \paren{1 - \paren{\frac{w_{\min}}{x}}^{2r}}^{s/r}
 + 2\delta \Delta^{2(s+2)}
 + 2 \epsilon_\net. \qedhere 
\end{align*}
\end{proof}

Let us now choose appropriate values of $r$ and $s$.

\begin{lemma} \label{lem:finite-r-s}
	Let $G$ be a connected $n$-vertex edge-weighted graph with maximum degree at most $\Delta \ge 2$
	 and all edge-weights in the interval $[w_{\min}, w_{\max}]$ with $0 < w_{\min} \le  w_{\max}$.
	Define $\wt \Delta = \Delta w_{\max}/w_{\min}$.
	Let $\delta = \mu_G(x, \infty)$ for some $x > 0$.
	Let $0 < \theta \le 1/\wt\Delta$ and suppose 
	\[
	    \delta \le \wt\Delta^{-10/\theta}.
	\]
	Let
	\[
		s = \floor{\frac{1}{\theta}}
		\quad\text{and}\quad
		r = \floor{\frac{1}{10} \log_{\wt\Delta} s}.
	\]
	Suppose $G$ has an  $r$-net with $\epsilon_\net n$ vertices. Then
	\[
	\mu_G[(1-\theta) x, x] \lesssim \epsilon_\net,
	\]
	where $\lesssim$ hides an absolute constant multiplicative factor.
\end{lemma}

\begin{proof}
We have an $r$-net with $\epsilon_\net n$ vertices.
The set of radius $r$ balls centered at the vertices of an $r$-net covers all $n$ vertices of a graph, and so $n \le 2\Delta^r \epsilon_\net n$, which gives $\epsilon_\net \gtrsim \Delta^{-r} \ge \wt\Delta^{-r} \ge s^{-1/10}$. 

We can assume that $x \le \Delta w_{\max}$ since all eigenvalues of $G$ lie in $[-\Delta w_{\max}, \Delta w_{\max}]$.
Let us bound each term in the conclusion of \cref{prop:finite-param}. 
For the first term,
\begin{align*}
(1-\theta)^{-2s} \paren{1 - \paren{\frac{w_{\min}}{x}}^{2r}}^{s/r} 
&\le \exp\paren{  4\theta s -  \frac{s}{r} \paren{\frac{w_{\min}}{x}}^{2r} }
\le \exp\paren{  4\theta s -  \frac{s}{r} \wt\Delta^{-2r} }
\\
&\le  \exp\paren{  4 - \frac{s}{r} s^{-1/5}}
\lesssim \frac{1}{s} \lesssim \epsilon_\net.
\end{align*}
For the second term of \cref{prop:finite-param}, we have
$\delta \le \wt\Delta^{-10/\theta} \le \wt \Delta^{-10s}$, and so
\[
\delta \Delta^{2(s+2)} 
\le \delta \Delta^{6s} 
\le \wt\Delta^{-4s}
\lesssim \frac{1}{s}
\lesssim \epsilon_\net .
\]
So \cref{prop:finite-param} gives  $\mu_G[(1-\theta)x, x] \lesssim \epsilon_\net$.
\end{proof}

Next we find small nets.

\begin{lemma}[Selecting a net {\cite[Lemma 4.2]{JTYZZ21}}]
\label{lem:net}
	Let $G$ be a connected $n$-vertex graph. Let $r$ be a nonnegative integer. Then $G$ has an $r$-net of size $\ceil{n/(r+1)}$. \qed 
\end{lemma}

\begin{proof}
	First replace $G$ by a spanning tree. Pick an arbitrary vertex as root. Let $v$ be a vertex furthest from the root. If $v$ is within distance $r$ from the root, then the root is already a one-element $r$-net.
	Otherwise, let $u$ be the vertex on the path from $v$ to the root and at distance exactly $r$ from $v$. 
	Add $u$ to the net and now delete the $\ge r+1$ vertices on the branch of the tree starting at $u$ (all deleted vertices are within distance $\le r$ from $u$ due to the choice of $v$). 
	Continue with the remaining tree.
\end{proof}

\begin{lemma}[Selecting a net in an expander]
\label{lem:net-expander}
Let $G$ be a connected $n$-vertex $c$-expander graph. 
Then for every $p \in [0,1]$ and positive integer $r \le c^{-1}\log(n/2)$, there exists an $r$-net of  size at most $((1-p)^{(1+c)^r} + p)n$.
\end{lemma}

\begin{proof}
Let $W_0$ be a random subset of $V(G)$ where each vertex is included independently with probability $p$. 
Let $W_1$ be the set of all vertices in $G$ that are at distance more than $r$ from $W_0$.
Then $W = W_0 \cup W_1$ is an $r$-net of $G$.

Since $G$ is a $c$-expander, every radius $r$ ball has at least $(1+c)^r$ vertices as long as $(1+c)^r \le n/2$.
So for every vertex $v$ of $G$, one has $\PP(v \in W_1) = (1-p)^{\abs{B_G(o, r)}} 
\le (1-p)^{(1+c)^r}$.
It follows that the expected size of $W$ is at most $(p +  (1-p)^{(1+c)^r}) n$, and thus there exists a net of at most this size.
\end{proof}

\begin{proof}[Proof of \cref{thm:finite-main}]
The result follows from \cref{lem:finite-r-s} and \cref{lem:net}. 
Note that if $\theta > 1/\wt\Delta$, then the conclusion 	$\mu_G[(1-\theta) x, x] \lesssim 1/\log_{\wt\Delta} (1/\theta)$ trivially follows from $\mu_G[(1-\theta) x, x] \le 1$. Thus we can assume $\wt\Delta \le \theta ^{-1} \le (1/10)\log_{\wt\Delta}(1/\delta)$, which allows us to apply \cref{lem:finite-r-s} with the $r, s, \epsilon_\net$ specified in the lemma. 
By \cref{lem:net}, we have $\epsilon_\net n \le \ceil{n/(r+1)} \lesssim n/r$ since it is assumed in \cref{thm:finite-expander} that $n \ge \log_{\wt\Delta}(1/\theta) = 10r$. 
So \cref{lem:finite-r-s} gives $\mu_G[(1-\theta)x,x] \lesssim \epsilon_\net \lesssim 1/r \lesssim 1/\log_{\wt\Delta}(1/\theta)$, as claimed.
\end{proof}

\begin{proof}[Proof of \cref{thm:finite-expander}]
The result follows from \cref{lem:finite-r-s} and \cref{lem:net-expander}.
Note that if $\theta > 1/\wt\Delta$,
then the conclusion $\mu_G[(1-\theta) x, x] \lesssim \theta^{c/40 \log\wt\Delta}$ follows trivially from $\mu_G[(1-\theta)x,x] \le 1$ (recall that $0 \le c \le 1$).
Thus we can assume $\wt\Delta \le \theta ^{-1} \le (1/10)\log_{\wt\Delta}(1/\delta)$, which allows us to apply \cref{lem:finite-r-s} with the $r, s, \epsilon_\net$ specified in the lemma. 
We have $r = \floor{(1/10)\log_{\wt\Delta} (1/\theta)} \le c^{-1}\log (n/2)$ (via the $n \ge 2\theta^{-c/(10\log \wt\Delta)}$ hypothesis in \cref{thm:finite-expander}).
Now apply \cref{lem:net-expander} with $p = (1+c)^{-r/2}$ to get an $r$-net with density
\[
\epsilon_\net \le (1-p)^{(1+c)^r} + p \le e^{-p(1+c)^r} + p \le e^{-(1+c)^{r/2}} + (1+c)^{-r/2} \lesssim e^{-cr/4} \le \theta^{-c/(40\log\wt\Delta)}.
\]
So \cref{thm:finite-expander} follows from  \cref{lem:finite-r-s}.
\end{proof}

\section{Unimodular random graphs} \label{sec:infinite}

We begin by recalling the notion of a unimodular random graph \cite{BS01,AL07}.
From now on, a ``graph'' is allowed to be infinite, though we only consider bounded degree graphs. 
A \emph{rooted graph} is a pair $(G,o)$ where $G$ is a graph and $o$ is a vertex in $G$.
The space of rooted graphs of degree at most $D$ can be endowed a distance: two rooted graphs have distance $2^{-r}$ where $r$ is the largest integer such that radius $r$ balls at the root are isomorphic as rooted graphs.
We consider a \emph{random rooted graph} $(G,o)$, with the understanding that the probability distribution should to be Borel measurable with respect to the topology given by the above metric (after identifying rooted graphs up to isomorphism classes).

A sequence $G_1, G_2, \dots$ of finite graphs of bounded degree is \emph{local convergent} (also known as \emph{Benjamini--Schramm convergent}) if, letting $o_n$ be a uniform random vertex of $G_n$, for every fixed $r\ge 0$, the random rooted finite graph $(B_{G_n}(o_n, r), o_n)$ convergences in distribution as $n \to \infty$.
The limit can be represented by a (possibly infinite) random rooted graph $(G,o)$ so that $(B_{G_n}(o_n, r), o_n)$ converges in distribution to $(B_G(o, r),o)$.

Any random rooted graph that is the limit of finite graphs enjoys a property known as unimodularity.
Given a rooted random graph $(G,o)$, let us first bias the distribution by the degree of $o$ in $G$ to obtain $(\wt G, \wt o)$ (i.e., for any function $f$ on root graphs, $\EE[f(\wt G,\wt o)] = \EE[\deg_G o]^{-1} \EE[f(G,o) \deg_G o]$)
Then, let $\wt x$ be a uniform random neighbor of $\wt o$.
As an illustrative special case, if $G$ is a finite graph and $o$ is a vertex chosen uniformly at random,
then $\wt o \wt x$ is a uniformly chosen edge along with an orientation.
A \emph{unimodular random graph} is a random rooted graph $(G,o)$ where $(\wt G, \wt o, \wt x)$ has the same distribution has $(\wt G, \wt x, \wt o)$.
This condition can be equivalently stated as the reversibility of a simple random walk.
It always holds for finite graphs with a uniformly chosen root, but does not necessarily hold for infinite rooted graphs even if we impose vertex transitivity---a famous example is the ``grandfather graph,'' which is a non-unimodular vertex-transitive infinite graph.
While every limit of finite graphs is unimodular, it is a major open problem whether every unimodular random graph arises as the limit of finite graphs; see \cite{AL07} for discussions.

We allow edge weights on $G$, although the reader is welcomed to assume unit edge weights for simplicity.
We can extend the definition of unimodularity by requiring that rooted graph isomorphisms also preserve edge weights. In the notation of the previous paragraph, $\wt x$ should be a uniform neighbor of $\wt o$ ignoring edge weights.

We shall apply unimodularity via the \emph{mass transport principle}: 
any nonnegative Borel measurable function $f_G(x,y)$ satisfies
\[
\EE \sum_{x \in V(G)} f_G(o,x) 
= 
\EE \sum_{x \in V(G)} f_G(x,o).
\]
The mass transport principle turns out to be equivalent to the reversibility characterization of unimodularity \cite{AL07}.
A word about measurability: we require that $(G,o,x) \mapsto f_G(o,x)$ is Borel measurable on the space of doubly rooted edge-weighted graphs.
Furthermore, it will be important to consider random functions $f_G(o,x)$ that depend Borel measurably on i.i.d.\ vertex labels of $G$.

\medskip

Let us now recall the definition of the spectral measure.
Given an edge-weighted graph $G$ with edge weight $w(x,y) = w(y,x)$ between $x,y \in V(G)$, we define the adjacency operator $A_G$ on $\ell^2(V(G))$, the Hilbert space of square summable functions on $V(G)$, by setting, for each $\psi \in \ell^2(V(G))$,
\[
(A\psi)(x) = \sum_{y \in V(G)} w(x,y) \psi(y).
\]
This extends the notion of the adjacency matrix of a finite graph.
In this paper, we only consider bounded degree graphs with bounded edge weights.
So $A_G$ is a bounded operator. 

The spectral theorem for bounded self-adjoint operators provides a spectral decomposition
\[
	A_G = \int x \, dP^G_x.
\]
where $E \mapsto P^G_E$ is a spectral measure. To be a \emph{spectral measure} means that $P^G_E$ satisfies the following properties:
\begin{enumerate}[(a)]
	\item for every measurable $E \subseteq \RR$, $P^G_E \colon \ell^2(V(G)) \to \ell^2(V(G))$ is an orthogonal projection;
	\item $P^G_\emptyset = 0$; $P^G_\RR$ is the identity;
	\item if $E_1, E_2, \dots$ are disjoint, and $E$ is their union, then $P^G_E = P^G_{E_1} + P^G_{E_2} + \cdots$;
	\item $P^G_{E_1\cap E_2} = P^G_{E_1} P^G_{E_2}$ for all $E_1, E_2$.
\end{enumerate}
The \emph{spectral subspace} $X^G_E$ is defined to be image of $P_E^G$.
In other words, $P_E^G$ is the orthogonal projection onto $X^G_E$.

Given a random rooted graph $(G,o)$ and a random closed subspace $X$ of $\ell^2(V(G))$ (the randomness of $X$ needs to be independent of $o$; in practice it will be a function of $G$ and its i.i.d.\ vertex labels), define
\[
\dim_G X = \EE \ang{1_o, \proj_X 1_o}.
\]
As an example, when $G$ is a finite graph and $o$ is a uniform vertex,
$\dim_G X$ is equal to the expected dimension of $X$ as a vector space divided by $\abs{G}$.

We define the \emph{spectral distribution} $\mu_G$ by setting, for every measurable $E \subseteq \RR$,
\[
\mu_G(E) = \EE \angs{1_o, P^G_E 1_o} = \dim_G X^G_E.
\]
This is a probability distribution on $\RR$.
It can be characterized by its moments, which correspond to the expected number of closed walks: for each integer $k \ge 0$, 
\begin{equation}\label{eq:unimod-return}
\int x^k \, d\mu_G(x) = \EE \angs{1_o, A_G^k 1_o},
\end{equation}
and note that $\angs{1_o, A_G^k 1_o}$ is the number of closed walks of length $k$ starting and ending at $o$ (if $G$ has edge weights, then we need to include weights in our walk count). 
For finite graphs, this agrees with the usual notion of the spectral distribution of the adjacency matrix.

\section{Spectral interlacing} \label{sec:interlacing}

The goal of this section is to prove an extension of the Cauchy eigenvalue interlacing theorem to unimodular random graphs (we refer the reader to \cite[Lemma 3.5, Theorem 3.9, Theorem 3.10]{BVInterlace} for further generalizations in the more abstract setting of tracial von Neumann algebras). 

\begin{theorem}[Spectral interlacing] \label{thm:interlacing}
	Let $(G,o)$ be a unimodular random edge-weighted graph with bounded degrees and bounded edge weights.
	Let $U \subseteq V(G)$ be a random subset of vertices whose distribution is independent of $o$.
	Let $H$ be obtained from $G$ by removing all edges incident to $U$.
	Then, for every $x \in \RR$,
	\[
		\abs{\mu_G(-\infty,x] - \mu_H(-\infty,x]} \le  \PP(o \in U).
	\]
\end{theorem}

\begin{corollary}
	Under the same setup as \cref{thm:interlacing}, for any $x \le y$,
	\[
	\abs{\mu_G[x,y] - \mu_H[x,y]} \le 2 \PP(o\in U).
	\]
\end{corollary}

To prove \cref{thm:interlacing}, we will use the polar decomposition theorem for bounded operators on Hilbert spaces \cite[Theorem VIII.3.11]{Conway}.
It says that for any bounded operator $S$, there is a partial isometry $T$ with initial space $(\ker S)^\perp$ and final space $\cl(\im S)$,
and so that $S = T \abs{S}$ where $\abs{S} := (S^*S)^{1/2}$.
Since $T$ is a partial isometry, $T^* T$ the orthogonal projection onto the initial space $(\ker S)^\perp$, and $TT^*$ is the orthogonal projection onto the final space $\cl(\im S)$.

Here is an extension of a basic fact about dimensions from linear algebra.

\begin{lemma} \label{lem:dim-ineq}
	Let $(G,o)$ be an edge-weighted unimodular random graph.
	Let $X$ and $Y$ be random closed subspaces of $\ell^2(V(G))$ (independent of $o$). Let $Z = X \cap Y^\perp$.  
	Then
	\[
	\dim_G X \le \dim_G Y + \dim_G Z.
	\]
\end{lemma}

\begin{proof}
	Let $Y' \le Y$ denote the image of the projection of $X$ on $Y$.
	Consider the polar decomposition of the operator $S = \proj_Y \proj_X$ as $S = T \abs{S}$, where $T$ is a partial isometry with initial space $(\ker S)^\perp = (X^\perp + Z)^\perp = X \cap Z^\perp$ and final space $Y'$. 
	So $T^*T =\proj_{(\ker S)^\perp} = \proj_{X \cap Z^\perp}$ and $TT^* = \proj_{\cl(\im S)}= \proj_{Y'}$.
	Applying the mass transport principle at step labeled $\mathsf{\scriptstyle(MTP)}$ below, we have
	\begin{align*}
	\dim_G X \cap Z^\perp 
	&= \EE \ang{1_o, \proj_{X \cap Z^\perp} 1_o}	
	= \EE \ang{1_o, T^* T 1_o}
	= \EE \ang{T 1_o, T1_o}
	\\
	&= \EE \sum_{v \in V(G)} \abs{\ang{1_v, T1_o}}^2
	\stackrel{\mathsf{\scriptscriptstyle(MTP)}}{=} \EE \sum_{v \in V(G)} \abs{\ang{1_o, T1_v}}^2
	= \EE \sum_{v \in V(G)} \abs{\ang{T^*1_o, 1_v}}^2 
	\\
	&= \EE \ang{T^*1_o, T^*1_o}
	= \EE \ang{1_o, TT^*1_o}
	= \EE \ang{1_o, \proj_{Y'} 1_o}
	\le \EE \ang{1_o, \proj_Y 1_o}
	= \dim_G Y.
	\end{align*}
	Since $Z \le X$, we have $X = (X \cap Z^\perp) \oplus Z$, and so 
	\[
	\dim_G X = \dim_G X \cap Z^\perp + \dim_G Z \le \dim_G Y + \dim_G Z. \qedhere
	\]
\end{proof}

\begin{lemma} \label{lem:local-large-spec}
	Let $(G,o)$ be an edge-weighted unimodular random graph with bounded degrees and bounded edge weights.
	Let $x \in \RR$.
	Let $X \subseteq \ell^2(V(G))$ be a random closed subspace (independent of $o$).
	If $\ang{\psi, A_G \psi}  > x$ for every unit vector $\psi \in X$,
	then 
	$
	\dim_G X \le \mu_G(x,\infty).
	$
\end{lemma}

\begin{proof}
	Any unit vector in the spectral subspace $X^G_{(-\infty,x]} = (X^G_{(x,\infty)})^\perp$ satisfies $\ang{\psi, A_G \psi} \le x$.
	Thus $X \cap (X^G_{(x,\infty)})^\perp = 0$.
	By \cref{lem:dim-ineq}, $\dim_G X \le \dim_G X^G_{(x,\infty)} = \mu_G(x,\infty)$.
\end{proof}

We are now ready to prove the spectral interlacing theorem.

\begin{proof}[Proof of \cref{thm:interlacing}]
	For any $\psi \in X^H_{(x,\infty)}$ that vanishes on $U$, we have 
	$\ang{\psi, A_G \psi} = \ang{\psi, A_H \psi} > x$. So by \cref{lem:local-large-spec}, we have
	\[
	\dim_G (X^H_{(x,\infty)} \cap \ell^2(V(G) \setminus U)) \le \mu_G(x,\infty).
	\]
	So, by \cref{lem:dim-ineq},
	\begin{align*}
	\mu_H(x,\infty) - \PP(o \in U) 
	&= \dim_G X^H_{(x,\infty)} - \dim_G \ell^2(U)
	\\
	&\le \dim_G (X^H_{(x,\infty)} \cap \ell^2(V(G) \setminus U)) \le \mu_G(x,\infty).
	\end{align*}
	Thus 
	\[
	\mu_G(x,\infty) - \mu_H(x,\infty) \le \PP(o \in U).
	\]
	The same proof (with $>x$ replaced by $\le x$, etc.) also shows $\mu_G(-\infty,x] - \mu_H(-\infty,x] \le \PP(o \in U)$.
	Together they imply $\abs{\mu_G(-\infty,x] - \mu_H(-\infty,x]} \le \PP(o \in U)$.
\end{proof}

We also need the following lemma.

\begin{lemma} \label{lem:sep-large-rad}
	Let $(G,o)$ be an edge-weighted unimodular random graph with bounded degrees and bounded edge weights.
	Let $s$ be a positive integer. 	Let $x \in \RR$.
	Let $U \subseteq V(G)$ be a random subset (independent of $o$) such that $U$ is $2(s+1)$-separated and $\lambda_1(B_G(u,s)) > x$ for all $u \in U$.
	Then $\PP(o \in U) \le \mu_G(x,\infty)$.
\end{lemma}

\begin{proof}
	For each $v \in V(G)$, if $v \in U$, then set $\psi_v \in \ell^2(V(G))$ to be the unique principal unit eigenvector with positive coordinates (i.e., the Perron vector) of $B_G(v,s)$, so that $\ang{\psi_v,\psi_v} = 1$, $\psi_v$ is supported on $B_G(v,s)$, and $\ang{\psi_v, A_G \psi_v} = \lambda_1(B_G(v,s)) > x$. 
	If $v \notin U$, then set $\psi_v = 0$. Let $X \le \ell^2(V(G))$ be the closure of the span of all $\psi_v$, $v \in V(G)$.
	
	Since $U$ is $2(s+1)$ separated, no vertex in the support of any $\psi_v$ is adjacent to a vertex in the support of another $\psi_u$. 
	All the nonzero $\psi_v$'s are orthogonal, and furthermore, for any nonzero finite linear combination $\psi = \sum_v c_v \psi_v$ with $c_v \in \CC$,
	we have $\ang{\psi, A_G \psi} = \sum_v \abs{c_v}^2 \ang{\psi_v, A_G \psi_v} > x \sum_v \abs{c_v}^2 = x \ang{\psi, \psi}$. 
	Hence, the same inequality holds for all nonzero $\psi \in X$.
	By \cref{lem:local-large-spec}, we have $\dim_G X \le \mu_G(x,\infty)$.
	
	Applying the mass transport principle at step labeled $\mathsf{\scriptstyle(MTP)}$ below, we have
	\begin{align*}
	\dim_G X 
	&= \EE \ang{1_o, \proj_X 1_o}
	\\
	&= \EE \sum_{v \in V(G)} \abs{\ang{1_o, \psi_v}}^2
	\stackrel{\mathsf{\scriptscriptstyle(MTP)}}{=} \EE \sum_{v \in V(G)} \abs{\ang{1_v, \psi_o}}^2
	= \EE \ang{\psi_o, \psi_o}
	= \PP(o \in U).
	\end{align*}
	Therefore, $\PP(o\in U) \le \mu_G(x,\infty)$.
\end{proof}

\section{Local selection} \label{sec:local}

In the proof of \cref{thm:2nd-eig}, we selected a large $s$-separated set as well as a small $r$-net. 
In the setting of random rooted graphs, it is important that our selection is independent of the root. 
For example, when the graph is an infinite path, all vertices are identical and unlabeled,
and so to select a subset of vertices (e.g., a large $s$-separated set or a small $r$-net), 
we need randomness.
We introduce i.i.d.\ random labels at the vertices.
For the random vertex selection to be well defined, it needs to depend Borel measurably on the vertex labels. 
The vertex selection also needs to respect unimodularity.
The simple way to ensure Borel measurability is to have each vertex $v$ decide its own inclusion based on the labels of vertices in its radius $R$ neighborhood for some finite $R$. This is the approach that we take.

This notion of local selection arises in the study of distributed graph algorithm. See Bernstein's recent survey \cite{Ber22} for a discussion of the connections between descriptive combinatorics and distributed algorithms. 
Imagine that the graph $G$ represents a network of computers, with each vertex representing a computer that can only communicate with its graph neighbors. 
Initially, the computers do not know anything about the network besides who its neighbors are.
At each time step, synchronously, each computer can send an arbitrary message to each of its neighbors. 
After $R$ steps, each computer can learn information about its radius $R$ neighborhood.
Each computer then must make a decision based on this local information (e.g., assign itself a color).

Our goal is to construct a random function $f$ on $V(G)$ (e.g., a vertex coloring). We say that $f$ is \emph{$R$-local} if it can be constructed in the following way: first assign an i.i.d.\ random label at each vertex, and then have each vertex $v$ simultaneously decide its value $f(v)$ deterministically using the isomorphism class of the rooted ball of radius $R$ at $v$ along with the random labels of vertices in this ball. 
We say that $f$ is \emph{local} if it is $R$-local for some finite constant $R$. If $f$ is $\{0,1\}$-valued, then it corresponds to selecting a random subset $U \subset V(G)$. For a unimodular random graph $(G,o)$ and such an $R$-local random subset $U$ of vertices, we define the \emph{density} of $U$ to be $\PP(o \in U)$. Note that here $U$ is independent of $o$.

Many easy graph algorithm tasks turn out to be hard locally.
For example, while it is trivial to find a maximal independent set via the greedy algorithm, a classic result of Linial~\cite{Lin92} implies that it is impossible to do so locally on an infinite path.
Remarkably, the minimum $R$ so that there exists an $R$-local maximal independent set in an $n$-cycle is $R = \Theta(\log^* n)$ \cite{CV86,Lin92}. This seminal result inspired much subsequent work on distributed graph algorithms.

In the proofs in \cref{sec:finite}, we used non-local greedy methods to select a large $s$-separated set and a small $r$-net. We need local alternatives. First, finding a large separated set is easy.

\begin{lemma}[Selecting a separated set] \label{lem:unimod-sep}
	Let $\Delta\ge 2, r, R$ be positive integers.
	Let $(G,o)$ be an edge-weighted unimodular random graph with maximum degree at most $\Delta$.
	Let $U \subseteq V(G)$ be an $R$-local random subset.
	Then there exists an $(r+R)$-local random $U_0 \subseteq U$ such that $U_0$ is $r$-separated in $G$ and has density $\PP(o \in U_0) \ge \Delta^{-r} \PP(o \in U)$.
\end{lemma}

\begin{proof}
	Assign each $v \in U$ an i.i.d. uniform label from $[0,1]$ (assign the label $0$ to all vertices outside $U$).
	Include $v$ in $U_0$ if $v \in U$ and its label is larger than all labels within distance $< r$ of $v$. 
	Then $U_0$ is $r$-separated, since otherwise we would have two vertices of $U$ within distance $<r$ each with a label larger than the other.
	Also $\PP(o \in U_0) = \abs{B_G(o, r-1)\cap U}^{-1} \PP(o \in U)  \ge \Delta^{-r} \PP(o \in U)$.
\end{proof}

It is more subtle how to locally pick a large $r$-net.
This is achieved in the next lemma.
The idea is to first randomly elect leaders (captains), who partition the graph into Voronoi cells. We then apply the the finite graph result (\cref{lem:net}) to select a good net in each Voronoi cell. To make sure that the selection is local, we restrict each Voronoi cell to a radius $R$ ball of the captain. We also add to our net any vertices too far from all captains, and this should only be a small fraction of the vertices.

\begin{lemma}[Selecting a net] \label{lem:unimod-net}
	For every positive integers $\Delta \ge 2$ and $r$, there exists a finite $R$ such that 
	if $(G,o)$ is an infinite connected edge-weighted unimodular random graph with maximum degree at most $\Delta$, 
	then $G$ has an $R$-local random $r$-net $W$ with density $\PP(o \in W) \le  1/r$.
\end{lemma}

\begin{proof}
	We set $p > 0$ to be a sufficiently small constant and $R$ a sufficiently large constant, both depending only on $\Delta$ and $r$, to specified later.

	Assign each vertex $v \in V(G)$ an i.i.d.\ uniform label $\xi_v$ from $[0,1]$.
	We call $v$ a \emph{captain} if $\xi_v \le p$. 
	For each vertex $v$, if there exists a captain within distance $\le R$ of $v$, then assign $v$ to \emph{report} to its nearest captain (breaking ties among nearest captains by taking the captain with the lower label). In particular, each captain reports to itself.
	If no captain lies within distance $\le R$ of $v$, then we say that $v$ is  \emph{unassigned}. Let $U$ be the set of unassigned vertices.  	
	We have 
	\[
		\PP(o \in U) = \PP(\text{no captain in $B_G(o, R)$}) = (1-p)^{\abs{B_G(o, R)}} \le (1-p)^R.
	\]
	
	Let $V_v$ denote the set of all vertices reporting to $v$ if $v$ is a captain,
	and let $V_v = \emptyset$ if $v$ is not a captain.
	Note that each $V_v$ induces a connected subgraph. Indeed, if $u$ reports to $v$, then every vertex on a shortest path from $u$ to $v$ also reports to $v$. 
	By \cref{lem:net}, we can select an $r$-net $W_v$ of $G[V_v]$ such that $\abs{W_v} = \ceil{\abs{V_v} / (r+1)}$ (set $W_v = \emptyset$ if $v$ is not a captain).
	We can make sure that the choice of $W_v$ is $R$-local by fixing a deterministic way to choose an $r$-net of desired size in each finite connected graph. 
	Let $W = \bigcup_{v \in V(G)} W_v$.
	
	By the mass transport principle,
	\[
	\EE \abs{V_o} 
	= \EE \sum_{v \in V(G)} \PP(v \in V_o) 
	= \EE \sum_{v \in V(G)} \PP(o \in V_v)
	\le 1
	\]
	since the sets $V_v$ are disjoint as $v$ ranges over $V(G)$.
	Applying the mass transport principle again,
	\begin{align*}
	\PP(o \in W) 
	&= \EE \sum_{v \in V(G)} \PP(o \in W_v) 
	= \EE \sum_{v \in V(G)} \PP(v \in W_o)
	= \EE \abs{W_o}
	\\
	&= \EE \ceil{\frac{\abs{V_o}}{r+1}}
	\le \EE \frac{\abs{V_o}}{r + 1/2} + \PP(0 < \abs{V_o} < 6r^2)
	\le  \frac{1}{r + 1/2} + \PP(0< \abs{V_o} < 6r^2),
	\end{align*}
	where the penultimate step comes from observing that $\ceil{n/(r+1)} > n/(r+1/2)$ implies $1 > n/(r+1/2) - n/(r+1) \ge n / (6 r^2)$.
	Finally, if $0 < \abs{V_o} < 6r^2$, then $o$ is captain (which occurs with probability $p$) and there is another captain with distance less than $12r^2$ from $o$ (probability $\le \abs{B_G(v, 12r^2-1)} p \le \Delta^{12r^2} p$). Thus 
	\[
	\PP(0 < \abs{V_o} < 6r^2) \le \Delta^{12r^2} p^2.
	\]
	
	Hence
	\[
	\PP(o \in U \cup W) \le \frac{1}{r+1/2} + \Delta^{12r^2} p^2 + (1-p)^R.
	\]
	By first choosing $p > 0$ sufficiently small and then $R$ sufficiently large (both functions of $r$ and $\Delta$) we can ensure that $\PP(o \in U \cup W) \le 1/r$.
\end{proof}

The proof of \cref{lem:net-expander} also carries over to the infinite setting.

\begin{lemma}[Selecting a net in an expander] \label{lem:unimod-net-expander}
Let $(G,o)$ be an infinite connected edge-weighted $c$-expander unimodular random graph with finite maximum degree.
Let $p \in [0,1]$.
Let $r$ be a positive integer.
Then $G$ has an $r$-local random $r$-net $W$ (where $W$ is independent of $o$) with density $\PP(o \in W) \le (1-p)^{(1+c)^r} + p$.
\end{lemma}

\begin{proof}
Let $W_0$ be a random subset of $V(G)$ where each vertex is included independently with probability $p$. 
Let $W_1$ be the set of all vertices in $G$ that are at distance more than $r$ from $W_0$.
Then $W = W_0 \cup W_1$ is an $r$-local random $r$-net.
We have $\PP(o \in W_0) = p$.
Since $G$ is a $c$-expander, every radius $r$ ball has at least $(1+c)^r$ vertices. 
So $\PP(o \in W_1) = (1-p)^{\abs{B_G(o, r)}} 
\le (1-p)^{(1+c)^r}$.
Thus $\PP(o \in W) \le p + (1-p)^{(1+c)^r}$.
\end{proof}

\section{Proof of spectral non-concentration for unimodular random graphs} 
\label{sec:unimod-pf}

Here is an analog of \cref{prop:finite-param} for infinite graphs.

\begin{proposition}
	 \label{prop:inf-param}
	Let $\Delta \ge 2$.
	Let $(G,o)$ be a connected infinite edge-weighted unimodular random graph with maximum degree at most $\Delta$ and all edge weights in the interval $[w_{\min}, w_{\max}]$ with $0 < w_{\min} \le  w_{\max}$.
	Let $r,s > 0$.
	Suppose $G$ has a random $r$-net $W$ ($W$ is independent of $o$) with $\PP(o \in W ) = \epsilon_\net$.
	Let $x \ge w_{\min}$ and $\delta = \mu_G(x,\infty)$.
	Let $\theta \in [0,1)$.
	Then
	\[
	\mu_G[(1 - \theta)x, x]
	\le
	(1-\theta)^{-2s} \paren{1 - \paren{\frac{w_{\min}}{x}}^{2r}}^{s/r}
	 + 2\delta \Delta^{2(s+2)}
	 + 2 \epsilon_\net .
	\]
\end{proposition}

We need an extension of \cref{lem:local-global} to the infinite setting.

\begin{lemma}[Local-global spectral comparison] \label{lem:unimod-local-global}
	Let $(G,o)$ be an edge-weighted random rooted graph with bounded degrees and bounded nonnegative edge-weights.
	Let $r$ be a positive integer. Then
	\[
	\int y^{2r} \, d\mu_G(y)
	\le
	\EE \sqb{\lambda_1(B_G(o,r))^{2r}}
	\]	
\end{lemma}

\begin{proof}
	As in the proof of \cref{lem:local-global}, we have
	\[
		\int y^{2r} \, d\mu_G(y)
		= \EE \angs{1_o, A_G^{2r} 1_o}
		= \EE \angs{1_o, A_{B_G(o, r)}^{2r} 1_o}
		\le \EE \sqb{\lambda_1(B_G(o, r))^{2r}}. \qedhere 
	\]
\end{proof}

\begin{proof}[Proof of \cref{prop:inf-param}]
Let
\[
U = \set{ v \in V(G) : \lambda_1 (B_G(v, s+1)) > x}.
\]
Note that $U$ is $(s+1)$-local. 
By \cref{lem:unimod-sep}, we can find a local random $2(s+2)$-separated $U_0 \subseteq U$ with $\PP(o \in U)
 \le \Delta^{2(s+2)} \PP(o \in U_0)$.
 By \cref{lem:sep-large-rad}, $\PP(o \in U_0) \le \mu_G (x,\infty) = \delta$.  So
\[
\PP(o \in U) \le \delta\Delta^{2(s+2)}.
\]
We are given that there is a random $r$-net $W \subseteq V(G)$ that is independent of $o$ and has density
\[
\PP(o \in W) \le \epsilon_\net.
\]
Let $H$ be obtained from $G$ by removing all edges incident to $U \cup W$.
For each $v \in V(G)$, the vertex complement of $B_H(v,s)$ in $B_G(v,s+1)$ forms an $r$-net of $B_G(v, s+1)$ (note that these balls are all finite).
Thus by \cref{lem:rad-drop}, for all $v \in V(G) \setminus U$,
\[
\lambda_1(B_H(v, s))^{2r}
\le
\lambda_1(B_G(v, s+1))^{2r} - w_{\min}^{2r}
\le x^{2r} - w_{\min}^{2r}.
\]
Also, if $v \in U$, then $v$ is an isolated vertex in $H$, and thus $\lambda_1(B_H(v, s))^{2r} = 0$.
Applying \cref{lem:unimod-local-global},
\[
\mu_H[(1 - \theta)x, x] (1-\theta)^{2s} x^{2s} 
\le
\int y^{2s} \, d\mu_H (y)
\le
\EE \sqb{\lambda_1( B_H(o, s))^{2s}}
\le 
(x^{2r} - w_{\min}^{2r})^{s/r}.
\]
Thus
\[
\mu_H[(1 - \theta)x, x] 
\le 
(1-\theta)^{-2s} \paren{1 - \paren{\frac{w_{\min}}{x}}^{2r}}^{s/r}
\]
By interlacing, \cref{thm:interlacing},
\begin{align*}
\mu_G[(1 - \theta)x, x]
&\le 
\mu_H[(1 - \theta)x, x] + 2 \PP(o \in U \cup W)
\\
&
\le
(1-\theta)^{-2s} \paren{1 - \paren{\frac{w_{\min}}{x}}^{2r}}^{s/r}
 + 2\delta \Delta^{2(s+2)}
 + 2 \epsilon_\net. \qedhere 
\end{align*}
\end{proof}

We have now finished proving all the necessary ingredients for extending the finite graph proof from \cref{sec:finite} to infinite unimodular random graphs.
This completes our proof of \cref{thm:unimod-main,thm:unimod-expander}
which are now identical to the proof of 
\cref{thm:finite-main,thm:finite-expander}
except that we apply \cref{prop:inf-param} instead of \cref{prop:finite-param},
and the infinite versions of the net selection lemmas (\cref{lem:unimod-net,lem:unimod-net-expander}) 
instead of the finite ones (\cref{lem:net,lem:net-expander}).

\section{Infinite regular expanders}\label{sec:non amenable}

Let us now prove \cref{thm:RegularExp}. For infinite regular expanders, we can skip the local-global step (\cref{lem:local-global,lem:unimod-local-global}) of the proof to obtain better bounds.

Here is an extension of \cref{lem:rad-drop} (net removal lowers spectral radius) to infinite graphs.

\begin{lemma} \label{lem:rad-drop-inf}
Let $(G,o)$ be a $d$-regular unimodular random graph with spectral norm $\rho$. Let $W$ be a random $r$-net of $G$ (independent of $o$).
Let $H$ be obtained by from $G$ by removing $W$. Then the spectral radius of $H$ is at most $(\rho^{2r} - 1)^{1/(2r)}$.
\end{lemma}

\begin{proof}
    As in the proof of \cref{lem:rad-drop}, we have that $A_H^{2r} \le A^{2r}_G - I$ entrywise as operators. Note that both sides have nonnegative entries. The consequence then follows from taking spectral norm on both sides.
\end{proof}

\begin{lemma} \label{lem:expander-net-rem-rad}
Let $(G,o)$ be an infinite $d$-regular unimodular random graph with spectral norm $\rho$. Suppose there exists an $r$-net of density $\delta > 0$. Let $\theta = 1 - (1-\rho^{-2r})^{1/(2r)}$. Then
\[
\mu_G[(1-\theta)\rho, \rho] \le \delta.
\]
\end{lemma}

\begin{proof}
Let $H$ be $G$ with the net removed. By \cref{lem:rad-drop-inf}, $\norm{H}^{2r} \le \norm{G}^{2r} - 1 = \rho^{2r} - 1$. And so $\norm{H} \le (1-\theta) \rho$. The result then follows from the spectral interlacing (\cref{thm:interlacing}).
\end{proof}

\begin{lemma}
   Every infinite $d$-regular unimodular random graph with spectral radius $\rho$ is an $c$-expander with $1 + c = d^2/\rho^2$.
\end{lemma}

\begin{proof}
For any finite subset $S$ of vertices, by the Cauchy--Schwarz inequality, almost surely,
\[
\abs{N(S)}^{-1} \le \norm{ \frac{A_G 1_S}{d\abs{S}}}_2^2 \le \rho^2 \norm{ \frac{1_S}{d\abs{S}}}_2^2 = \frac{\rho^2}{d^2 \abs{S}}.
\]
So $\abs{N(S)} \ge (d/\rho)^2\abs{S}$.
Thus $G$ is an $c$-expander.
\end{proof}

\begin{proof}[Proof of \cref{thm:RegularExp}]
Throughout, we shall assume that $\theta > 0$ is sufficiently small (i.e., $\theta$ less than some constant depending on $d$ and $\epsilon$). Let 
\[
r = \ceil{(1-\varepsilon) \frac{\log(1/\theta)}{2\log \rho}}.
\]
Then
\[
(1-\rho^{-2r})^{1/(2r)}
\le \exp\paren{\frac{-\rho^{-2r}}{2r}} \le \exp(-\theta^{1-\varepsilon/2}) \le 1- \theta.
\]
Setting
\[
p = \theta^{(1-2\varepsilon) \paren{\frac{\log d}{\log \rho} - 1}},
\]
we find that
\[
(1-p)^{(d/\rho)^{2r}}
\le \exp\paren{ - p \paren{\frac{d}{\rho}}^{2r} }
\le \exp\paren{ - \theta^{-\varepsilon \paren{\frac{\log d}{\log \rho} - 1}}} \le p.
\]
By \cref{lem:unimod-net-expander}, there exists an $r$-net of density at most $2p$. By \cref{lem:expander-net-rem-rad}, one has $\mu_G[(1-\theta) \rho, \rho] \le 2p$, which implies the desired result (after changing $\epsilon$ appropriately).
\end{proof}

\subsection*{Connections to return times of random walks}
Let us relate the non-concentration at near the top of the spectrum of an infinite regular unimodular graph to estimates on return times. Consider an infinite $d$-regular unimodular random graph $(G,o)$ with spectral radius $\rho$. Let $p_k$ be the probability that a simple random walk starting at $o$ returns to $o$ after $k$ steps. Recall from \cref{eq:unimod-return} that
\[
d^k p_k = \int t^k d\mu_G(t).
\]
We have 
\[
\lim_{k \to \infty} p_{2k}^{1/(2k)} = (\rho/d)^2.
\]
It turns that that the rate at which $p_{2k} (d/\rho)^{2k}$ decays with $k$ is closely related to the concentration of the spectrum of $G$ near $\rho$.

\begin{proposition}\label{prop:return-decay}
Let $(G,o)$ be a $d$-regular unimodular random graph with spectral radius $\rho$. Let $p_k$ be the probability that a simple random walk starting at $o$ returns to $o$ after $k$ steps.
\begin{enumerate}
    \item [(a)] If there are constants $C, \alpha, \rho > 0$ such that 
    \[
    p_{2n} \le C (\rho/d)^{2n} n^{-\alpha} \quad \text{ for all }n,
    \]
    then
    \[
    \mu_G[(1-\theta)\rho, \rho] \lesssim_{\alpha} C \theta^\alpha \quad \text{ for all $0 < \theta < 1/2$}.
    \]
    \item [(b)] If there are constants $C, \alpha > 0$ such that
    \[
    \mu_G[(1-\theta)\rho, \rho] \lesssim C \theta^\alpha \quad \text{ for all $0 < \theta < 1/2$},
    \]
    then for every $\epsilon > 0$, there is some constant $C'$ (which is allowed to depend on $(G,o)$, $C$, $\alpha$, $\epsilon$) such that
    \[
    p_{2n} \le C' (\rho/d)^{2n} n^{-\alpha + \epsilon} \quad \text{ for all }n.
    \]
\end{enumerate}
\end{proposition}

\begin{proof}
(a)
Let $n = \ceil{1/(2\theta)}$. We have
\[
\mu_G[(1-\theta)\rho, \rho] (1-\theta)^{2n} \rho^{2n} 
\le \int t^{2n} d\mu_G(t) 
= d^{2n} p_{2n}
\le C \rho^{2n} n^{-\alpha}.
\]
Thus
\[
\mu_G[(1-\theta)\rho, \rho] \le C (1-\theta)^{-2n} n^{-\alpha} \lesssim C n^{-\alpha} \lesssim_{\alpha} C \theta^{-\alpha}.
\]

(b) 
Let $\theta = n^{-1 + \epsilon / \alpha}$. We have
\begin{align*}
\int_0^\rho t^{2n} d \mu_G(t) 
&\le (1-\theta)^{2n} \rho^{2n} + \rho^{2n} \mu_G[(1-\theta)\rho, \rho]
\\&\lesssim_C  \rho^{2n} \paren{(1-\theta)^{2n} + \theta^\alpha}
\\& \lesssim_{C, \alpha, \epsilon} \rho^{2n} n^{-\alpha + \epsilon}.
\end{align*}
If $p_n = 0$ for all odd $n$, then
\[
\int (-t)^n d \mu_G(t) = \int t^n d \mu_G(t) = d^n p_n
\]
for all nonnegative integers $n$ (noting that when $n$ is odd, both sides are equal to zero as $p_n = 0$). Thus, by the Stone--Weierstrass and Riesz representation theorems and the fact that $\mu_G$ is supported on $[-\rho, \rho]$, we see that $\mu_G$ is symmetric about zero. And so
\[
d^{2n} p_{2n} 
= \int_{-\rho}^\rho t^{2n} d \mu_G(t) 
= 2\int_{0}^\rho t^{2n} d \mu_G(t) 
\lesssim_{C, \alpha, \epsilon} \rho^{2n} n^{-\alpha + \epsilon}.
\]
On the other hand, if there is some odd $k$ with $p_k > 0$, then for any positive integer $n$ we have
\[
p_{2n} 
\le p_k^{-1} p_{2n+k} 
= d^{-2n-k} p_k^{-1} \int_{-\rho}^{\rho} t^{2n+k}d\mu_G(t) 
\le 2 d^{-2n} p_k^{-1} \int_0^{\rho} t^{2n}d\mu_G(t),
\]
and we can upper bound the final integral by $O_{C, \alpha, \epsilon}(\rho^{2n} n^{-\alpha + \epsilon})$ as earlier.
\end{proof}

\subsection*{Connnections to the rapid decay property}

We can show that in the case of spectral nonconcentration for adjacency operators on Cayley graphs, our results are optimal. For this, we need an auxiliary notion.
For a group $\Gamma$, let $\rho\colon \Gamma\to \mathcal{U}(\ell^{2}(\Gamma))$ be the left regular representation given by
\[(\lambda(g)\xi)(h)=\xi(g^{-1}h) \textnormal{ for all $g,h\in \Gamma,\xi\in\ell^{2}(\Gamma)$}.\]
\begin{definition}\label{defn: RDP}
Let $\Gamma$ be a finitely generated group with finite generating set $S$. We let $|g|_{S}$ be the word length of an element $g\in \Gamma$. We say $\Gamma$ has the \emph{rapid decay property} if there are constants $C,\alpha>0$ so that 
\[\left\|\sum_{g\in \Gamma}f(g)\lambda(g)\right\|\leq C\left(\sum_{g\in\Gamma}|f(g)|^{2}(|g|_{S}^{2}+1)^{2\alpha}\right)^{1/2}\]
for all finitely supported $f\colon \Gamma\to \CC$.
\end{definition}

A good example for intuition is the example $\Gamma=\ZZ^{d}$. In this case, by Parseval's theorem, the left-hand side of the above inequality is the supreumum norm of the function  $(S^{1})^{d}\to \CC$ given by
\[\zeta\mapsto \sum_{n\in \ZZ^{d}}f(n)\prod_{j=1}^{d}\zeta_{j}^{n_{j}}.\]
The fact that $\ZZ^d$ satisfies the rapid decay property is an exercise in Fourier analysis and the Cauchy--Schwarz inequality. 

For discrete nonabelian groups, the situation is more difficult and the first nontrivial example was provided by Haagerup in \cite{HaagRDP} (for topological groups, a example was previously exhibited in \cite{HerzRDP}). He was interested in this property in the context of approximation proprieties of the reduced $C^{*}$-algebra of free groups and proved that free groups satisfy the rapid decay property. Since then, the rapid decay properties has been established for virtually nilpotent groups, hyperbolic groups, and many others. See \cite{ChatterjiRDPSurvey} for a good introduction to this topic.
The name ``rapid decay property'' comes from the fact that the estimate in the above definition allows one to define a algebra of bounded convolution operators from ``Schwartz functions'' on $\Gamma$, and this algebra  serves as the analogue of Fourier transforms of $C^{\infty}$-functions on $(S^{1})^{d}$ for the case of $\ZZ^{d}$.

The rapid decay property is relevant to use because of the following result of Chatterji, Pittet, Saloff-Coste (we remark that other applications of the rapid decay property to random walks include \cite[Example 2]{PNnonunique}).
\begin{theorem}[{\cite[Theorem 1.3]{ChatterjiRDPSurvey} and \cite[Section 7]{CPSPRDLie}}]  \label{thm:rd-return}
Let $\Gamma$ be a finitely generated group with the rapid decay property and let $\alpha$ be as in Definition \ref{defn: RDP}. Let $\nu$ be a finitely supported and symmetric probability measure on $G$.
Consider the random walk taking independent steps from $\nu$.
Let $p_n$ be the probability that it returns to the starting vertex after $n$ steps. Then
\[p_{2n}\gtrsim_{\nu}\rho^{2n}n^{-2\alpha}.\]
\end{theorem}

Then, combined with \cref{prop:return-decay}, we deduce the following. (Technically, \cref{prop:return-decay} is stated for unweighted graphs, but the same proof applies with $\nu$ weights.)

\begin{corollary}
Assume the hypotheses of \cref{thm:rd-return}.
Let $\mu$ be the unique measure on $\RR$ so that 
\[p_{k}=\int t^{k}\,d\mu(t)\]
for all nonnegative integers $k$. Then for every $\epsilon > 0$, 
\[\mu([(1-\theta)\rho,\rho])\gtrsim_{\nu, \alpha, \epsilon}\theta^{2\alpha + \epsilon} \quad \text{for all } 0 < \theta < 1/2.\]
\end{corollary}

Note that this says the polynomial upper bound we exhibit in \cref{thm:unimod-main,thm:RegularExp} cannot be improved in general, and thus our results are optimal in the general case. It is an unpublished result of  Chatterji, Pittet, Saloff-Coste (see the discussion following Theorem 1.3 in \cite{ChatterjiRDPSurvey}) that for a general discrete group, 
\[p_{2n}\lesssim_{\nu} \rho^{2n}n^{-1}\]
and thus, by Proposition \ref{prop:return-decay}
\[\mu([(1-\theta)\rho,\rho])\lesssim_\nu \theta.\]
This is better than what we can obtain from \cref{thm:unimod-main,thm:RegularExp}.
However it is unclear to us if their arguments can be adapted to the more general unimodular random graph case. Additionally, the argument they use is a pure existence argument which produces an ineffective constant. With some bookkeeping, one can produce explicit constants in our results. 
For a finitely generated amenable group, the rapid decay property is equivalent to polynomial growth \cite[Proposition B]{JoliRDP}. One should compare the results of Chatterji, Pittet, Saloff-Coste and the discussion following Theorem \ref{thm:decay on return probs intro}.


\end{document}